\documentclass[12pt]{amsart}

\usepackage{color}
\usepackage[all]{xy}
\usepackage{mathrsfs}
\usepackage{amsmath, amssymb,ascmac}
\usepackage{graphics}
\usepackage{graphicx}  
\usepackage{geometry}
\newtheorem{theo}{Theorem}[section]

\newtheorem{cor}[theo]{Corollary}

\newtheorem{lemma}[theo]{Lemma}

\newtheorem{defi}[theo]{Definition}
\newtheorem{rem}[theo]{Remark}
\newtheorem{example}[theo]{Example}

\newtheorem{nota}{Notation}

\newtheorem*{main}{Main Result}

\newcommand{\End}{\operatorname{End}}
\newcommand{\Hom}{\operatorname{Hom}}
\renewcommand{\Im}{\operatorname{Im}}
\newcommand{\Ker}{\operatorname{Ker}}
\newcommand{\Coker}{\operatorname{Cok}\nolimits}
\newcommand{\id}{\operatorname{id}}

\newcommand{\Cone}{\operatorname{Cone}}

\renewcommand{\mod}{\mbox{-}\mathrm{mod}}

\newcommand{\smod}{\mbox{-}\underline{\mathrm{mod}}}

\newcommand{\proj}{\mbox{-}\mathrm{proj}}

\newcommand{\perf}{\mathrm{perf}}
\newcommand{\CN}{\mathcal{CN}}

\newcommand{\hExt}{\operatorname{\widehat{Ext}}}

\newcommand{\Hh}{\operatorname{H}\nolimits}
\newcommand{\hHh}{\operatorname{\widehat{H}}\nolimits}

\newcommand{\hHH}{\operatorname{\widehat{HH}}}

\newcommand{\gldim}{\operatorname{gl.dim}}
\newcommand{\injdim}{\operatorname{inj.dim}}
\newcommand{\projdim}{\operatorname{proj.dim}}

\renewcommand{\L}{\Lambda}
\newcommand{\G}{\Gamma}
\newcommand{\Le}{\Lambda^\textrm{e}}
\newcommand{\Ge}{\Gamma^\textrm{e}}

\newcommand{\D}{\mathcal{D}}
\newcommand{\K}{\mathcal{K}}

\numberwithin{equation}{section}

\begin{document}
%%%%%%%%%%%%%%%%%%%%%%%%%%%%%%%%%%%%%%%%%%%%%
%    Title        
%%%%%%%%%%%%%%%%%%%%%%%%%%%%%%%%%%%%%%%%%%%%%
\title[Characterization of eventually periodic modules]{Characterization of eventually periodic modules in the singularity categories}
%%%%%%%%%%%%%%%%%%%%%%%%%%%%%%%%%%%%%%%%%%%%%
%    Title        
%%%%%%%%%%%%%%%%%%%%%%%%%%%%%%%%%%%%%%%%%%%%%

%%%%%%%%%%%%%%%%%%%%%%%%%%%%%%%%%%%%
%       author one information
%%%%%%%%%%%%%%%%%%%%%%%%%%%%%%%%%%%%
\author[S. Usui]{Satoshi Usui}
\address{%
Graduate School of Science, Department of Mathematics,
Tokyo University of Science, 
1-3, Kagurazaka, Shinjuku-ku, 
Tokyo, 162-8601,
Japan}
\email{usui.satoshi.g@gmail.com} 

%%%%%%%%%%%%%%%%%%%%%%%%%%%%%%%%%%%%
%       author one information
%%%%%%%%%%%%%%%%%%%%%%%%%%%%%%%%%%%%

\thanks{}

 %Representations of associative Artinian rings
%(secondary)


\date{}

\dedicatory{}

%%%%%%%%%%%%%%%%%%%%%%%%%%%%%%%%%%%%
%           Section ↓
%%%%%%%%%%%%%%%%%%%%%%%%%%%%%%%%%%%%
%\section{ } 

%%%%%%%%%%%%%%%%%%%%%%%%%%%%%%%%%%%%
%           Section ↑
%%%%%%%%%%%%%%%%%%%%%%%%%%%%%%%%%%%%

\maketitle

%%%%%%%%%%%%%%%%%%%%%%%%%%%%%%%%%%%%
%           Abstract
%%%%%%%%%%%%%%%%%%%%%%%%%%%%%%%%%%%%
\begin{abstract}
The singularity category of a ring makes only the modules of finite projective dimension vanish among the modules, so that the singularity category is expected to characterize a homological property of modules of infinite projective dimension.
In this paper, among such modules, we deal with eventually periodic modules over a left artin ring, and, as our main result, we characterize them in terms of morphisms in the singularity category.
As applications, we first prove that, for the class of finite dimensional algebras over a field, being eventually periodic is preserved under singular equivalence of Morita type with level.
Moreover, we determine which finite dimensional connected Nakayama algebras are eventually periodic when the ground field is algebraically closed.
\end{abstract}
%%%%%%%%%%%%%%%%%%%%%%%%%%%%%%%%%%%%
%           Abstract
%%%%%%%%%%%%%%%%%%%%%%%%%%%%%%%%%%%%

%\tableofcontents

%%%%%%%%%%%%%%%%%%%%%%%%%%%%%%%%%%%%
%           Introduction ↓
%%%%%%%%%%%%%%%%%%%%%%%%%%%%%%%%%%%%
\section*{Introduction} \label{intro}
Throughout this paper, let $k$ be a field, and we assume that all rings are associative and unital. 
By a module, we mean a left module.

The {\it singularity category} $\D_{\rm sg}(R)$ of a left Noetherian ring $R$ is defined to be the Verdier quotient of the bounded derived category of finitely generated $R$-modules by the full subcategory of perfect complexes (see \cite{Buch86,Orlov04}), and it  provides a homological measure of singularity of $R$ in the following sense: the global dimension of $R$ is finite if and only if $\D_{\rm sg}(R) = 0$. 
This follows from the fact that a finitely generated $R$-module has finite projective dimension if and only if it is isomorphic to $0$ in $\D_{\rm sg}(R)$. 
In view of these facts, it is expected that the singularity categories capture homological properties of rings of infinite global dimension and modules of infinite projective dimension.

This paper is devoted to studying modules with infinite projective dimension, and we focus on {\it eventually periodic modules} over a left artin ring, that is, finitely generated modules whose minimal projective resolutions have infinite length and become periodic from some step. 
Remark that eventually periodic modules in our sense have infinite projective dimension, while those in the other papers are not necessarily of infinite projective dimension (see Remark \ref{rem_1}). 
The eventual periodicity of modules has been studied for a long time 
(see \cite{Avramov-Gasharov-Peeva_1989,Bergh_2006,Croll_2013,Eisenbud_1980,kupper2010two,Usui21_No02}).
In \cite{Usui21_No02}, the author investigated eventually periodic modules  over a finite dimensional Gorenstein algebra $\L$ over an algebraically closed field. 
It was proved in \cite[Proposition 3.4]{Usui21_No02} that a finitely generated $\L$-module $M$ is eventually periodic if and only if there exists an invertible homogeneous element of positive degree in the {\it Tate cohomology ring} of $M$, which is the $\mathbb{Z}$-graded ring  whose $i$-th component is given by the morphism space in $\D_{\rm sg}(\L)$ from $M$ to $M[i]$, where the multiplication is given by the Yoneda product.
Hence the following question naturally occurs.
Does the same characterization of eventually periodic modules as above hold for more general rings?

There are three aims of this paper.
Our first aim is to prove our main result, which gives an affirmative answer to the above question.
More concretely, we characterize the eventual periodicity of modules over a left artin ring as in \cite[Proposition 3.4]{Usui21_No02}:

%
%%%%%%%%%%%%%%%     ↓ statement     %%%%%%%%%%%%%%%%%%%
%
\begin{main}[Corollary \ref{claim_13_1}] \label{intro_1} 
Let $R$ be a left artin ring. 
Then the following conditions are equivalent for a finitely generated $R$-module $M$.
\begin{enumerate}
    \item $M$ is eventually periodic.
    \item The projective dimension of $M$ is infinite, and the Tate cohomology ring of $M$  has an invertible homogeneous element of positive degree.  
\end{enumerate}
In this case, there exists an invertible homogeneous element in the Tate cohomology ring of $M$ whose degree equals the period of some periodic syzygy $\Omega_R^{n}(M)$ with $n \geq 0$.
\end{main}
%%%%%%%%%%%%%%%     ↑ statement     %%%%%%%%%%%%%%%%%%%

We remark that the degree of the invertible element in the last statement of the main result does not depend on the choice of a periodic syzygy, because all the periodic syzygies of an eventually periodic module have the same period.
The main result will be obtained from the result characterizing the existence of  invertible homogeneous elements in  Tate cohomology rings over a left Noetherian ring (see Theorem \ref{claim_13}).

Next, we work with finite dimensional eventually periodic $k$-algebras, that is, finite dimensional $k$-algebras $\L$ that are eventually periodic as $(\L \otimes_k \L^{\rm op})$-modules. 
Then, as our second aim, we show that, for the class of finite dimensional $k$-algebras, being eventually periodic is invariant under {\it singular equivalence of Morita type with level} (see Theorem \ref{claim_15}).
The notion of singular equivalences of Morita type with level  was introduced by Wang \cite{Wang_2015}, and it is known that  such equivalences induce triangle equivalences between singularity categories (see Remark \ref{rem_2}). 
Recently, singular equivalences of Morita type with level have been intensively studied (see \cite{Chen-Liu-Wang_2020,Dalezios_2021,Qin_2021,Skart_2016,Wang_2015}).
In particular, Skarts\ae terhagen \cite{Skart_2016}, Qin \cite{Qin_2021} and Wang \cite{Wang_2015,Wang_2019} have discovered invariants under singular equivalence of Morita type with level.

Our last aim is to give a necessary and sufficient condition for a finite dimensional connected Nakayama $k$-algebra to be eventually periodic when the ground field $k$ is algebraically closed (see Corollary \ref{claim_30}).
The key point is to classify such algebras up to singular equivalence of Morita type with level. 

This paper is organized as follows. 
%%%
%%%　　 
%%%
In Section 1, we recall basic terminology and facts which are used in this paper.
%%%
%%%　　 
%%%
In Section 2, we prove our main result.
%%%
%%%　　 
%%%
In Section 3, we apply our main result to prove that being eventually periodic is preserved under singular equivalence of Morita type with level and to give a criterion for a Nakayama algebra to be eventually periodic.

%%%%%%%%%%%%%%%%%%%%%%%%%%%%%%%%%%%%
%           Introduction ↑
%%%%%%%%%%%%%%%%%%%%%%%%%%%%%%%%%%%%

%%%%%%%%%%%%%%%%%%%%%%%%%%%%%%%%%%%%
%           Section ↓
%%%%%%%%%%%%%%%%%%%%%%%%%%%%%%%%%%%%
\section{Preliminaries} \label{Preliminaries}

For a ring $R$, we denote by $R\mod$ the category of finitely generated  $R$-modules, by  $R\proj$ the full subcategory of $R\mod$ consisting of finitely generated projective $R$-modules and by $\gldim R$ the global dimension of $R$. 
For an $R$-module $M$, we denote by $\projdim_{R}M$ (resp.\ $\injdim_{R}M$) the projective (resp.\ injective) dimension of $M$. 
For an algebra $\L$ over a commutative ring $S$, we denote by $\Le$ its enveloping algebra $\L \otimes_{S} \L^{\rm op}$, where $\L^{\rm op}$ is the opposite algebra of $\L$.
Then $\Le$-modules can be identified with $\L$-bimodules (on which $k$ acts centrally).

%%%%%%%%%%%%%%%%%%%%%%%%%%%%%%%%%%%%%%%%%%%%%%%%%%
\subsection{Complexes}
%%%%%%%%%%%%%%%%%%%%%%%%%%%%%%%%%%%%%%%%%%%%%%%%%%
In this subsection, we recall basic terminology related to complexes.
Let $R$ be a ring, and let $X_\bullet$ and $Y_\bullet$ be (chain) complexes of $R$-modules, where $X_\bullet$ is of the form
\[
\cdots \rightarrow X_{i+1} \xrightarrow{d_{i+1}^{X}} X_{i}  \xrightarrow{d_{i}^{X}} X_{i-1} \rightarrow \cdots.
\]
For each integer $i$, we denote by $\pi_i=\pi_i^{X}$ the canonical epimorphism $X_i \rightarrow \Omega_{i}(X_\bullet):= \Coker d^X_{i+1}$.
The {\it $i$-th shift} $X_\bullet[i]$ of $X_\bullet$ is the complex given by $(X_\bullet[i])_j = X_{j-i}$ and  $d^{X[i]} = (-1)^{i}d^{X}$.
We say that $X_\bullet$ is {\it left ({\rm resp.}\ right) bounded} if $X_i = 0$ for all $i \gg 0$ (resp.\ $i \ll 0$) and that $X_\bullet$ is {\it bounded} if it is both left and right bounded.
A {\it chain map} from $X_\bullet$ to $Y_\bullet$ is a family $(\varphi_i:X_i\rightarrow Y_i)_{i\in \mathbb{Z}}$ of morphisms of $R$-modules such that $d^{Y} \varphi = \varphi\,d^X$.
Let $\varphi$ and $\psi$ be two chain maps from $X_\bullet$ to $Y_\bullet$.
We say that the two chain maps $\varphi$ and $\psi$ are {\it homotopic}, denoted by $\varphi \sim \psi$, if there exists a family $(\sigma_i:X_i\rightarrow Y_{i+1})_{i\in \mathbb{Z}}$ of morphisms of $R$-modules such that $\varphi- \psi = d^Y \sigma + \sigma d^X$. 
Such a family is called  a {\it homotopy from $\varphi$ to $\psi$}.
The chain map $\varphi: X_\bullet \rightarrow Y_\bullet$ is called a {\it homotopy equivalence} if there exists a chain map $\psi:Y_\bullet \rightarrow X_\bullet$ such that $\psi\varphi\sim\id_X$ and $\varphi\psi\sim\id_Y$.
The {\it  cone} $\Cone(\varphi)_\bullet$ of $\varphi: X_\bullet \rightarrow Y_\bullet$ is the complex given by $\Cone(\varphi)_i:=Y_i \oplus X_{i-1}$ and 
\[ d^{\Cone(\varphi)}_i := \begin{pmatrix}d^{Y}_{i}&\varphi_{i-1}\\0&-d^{X}_{i-1}\end{pmatrix}:Y_i \oplus X_{i-1} \rightarrow Y_{i-1} \oplus X_{i-2},\]
where we view an element of $Y_i \oplus X_{i-1}$ as a column vector.
The complex $X_\bullet$ is called {\it contractible} if $\id_X \sim 0$. 
It is well-known that a chain map is homotopy equivalent if and only if its  cone is contractible. 
The {\it $i$-th homology group} of $X_\bullet$ is the quotient $\Hh_i(X_\bullet) := \Ker d^X_i/\Im d^X_{i+1}$.
We call $\varphi:X_\bullet \rightarrow Y_\bullet$ a {\it quasi-isomorphism} if the induced morphism $\Hh_i(\varphi):\Hh_i(X_\bullet)\rightarrow \Hh_i(Y_\bullet)$ is an isomorphism for all integers $i$.

%%%%%%%%%%%%%%%%%%%%%%%%%%%%%%%%%%%%%%%%%%%%%%%%%%
\subsection{Singularity categories} \label{Preliminaries_2}
%%%%%%%%%%%%%%%%%%%%%%%%%%%%%%%%%%%%%%%%%%%%%%%%%%
In the rest of Section \ref{Preliminaries}, let $R$ be a left Noetherian ring.

In this subsection, we review some facts on singularity categories from \cite{Buch86,Neeman_Book,Zimme14Book}.
Let $\D^{\rm b}(R\mod)$ be the bounded derived category of $R\mod$. 
It is well-known that there exists a triangle equivalence   $\K^{\rm -, b}(R\proj)\rightarrow\D^{\rm b}(R\mod)$, where $\K^{\rm -, b}(R\proj)$ stands for the homotopy category of right bounded complexes of finitely generated projective $R$-modules with bounded homology (see \cite{Zimme14Book}).
Recall that an object $X_\bullet$ of $\D^{\rm b}(R\mod)$ is called {\it perfect} if it is isomorphic in $\D^{\rm b}(R\mod)$ to a bounded complex of finitely generated projective $R$-modules.
The full subcategory $\perf(R)$ of $\D^{\rm b}(R\mod)$ consisting of perfect complexes becomes a thick subcategory, that is, a triangulated subcategory that is closed under direct summands (see \cite[Section 1.2]{Buch86}).
The triangle equivalence $\K^{\rm -, b}(R\proj) \rightarrow \D^{\rm b}(R\mod)$ restricts to a triangle equivalence $\K^{\rm b}(R\proj) \rightarrow \perf(R)$, where $\K^{\rm b}(R\proj)$ denotes the homotopy category of bounded complexes of finitely generated  projective $R$-modules.
Following \cite{Orlov04}, we call the Verdier quotient  $\D_{\rm sg}(R) := \D^{\rm b}(R\mod)/\perf(R)$ the {\it singularity category} of $R$ (cf.\ \cite{Buch86}).
Recall that the shift functors $[1]$ on $\D^{\rm b}(R\mod)$ and $\D_{\rm sg}(R)$ are induced by shift of complexes.
It follows from the definition of singularity categories that the object class of $\D_{\rm sg}(R)$ is the same as $\D^{\rm b}(R\mod)$, and the morphisms from $X_\bullet$ to $Y_\bullet$ are the equivalence classes $\left[X_\bullet \leftarrow Z_\bullet \rightarrow Y_\bullet \right]$ of fractions $X_\bullet \leftarrow Z_\bullet \rightarrow Y_\bullet $ of morphisms in $\D^{\rm b}(R\mod)$ such that the cone of $X_\bullet \leftarrow Z_\bullet$ is perfect.
Since $\perf(R)$ is  thick, it follows that $\left[X_\bullet \leftarrow Z_\bullet \rightarrow Y_\bullet \right]$ is an isomorphism in $\D_{\rm sg}(R)$ if and only if the cone of $Z_\bullet \rightarrow Y_\bullet$ is perfect (cf. \cite[Proposition 2.1.35]{Neeman_Book}).

Recall that the {\it stable category} $R\smod$ of $R\mod$ is the category whose objects are the same as $R\mod$ and morphisms are given by
$\underline{\Hom}_{R}(M, N) := \Hom_{R}(M, N)/\mathcal{P}(M, N),$
where $\mathcal{P}(M, N)$ is the abelian subgroup of morphisms from $M$ to $N$ factoring through a projective module.
There is the so-called {\it loop space functor}  $\Omega_{R}: R\smod \rightarrow R\smod$ (see \cite{Buch86,Heller_1960}), which takes a module $M$ in $R\mod$ to the kernel $\Omega_{R}(M) = \Ker(P\rightarrow M)$ of an epimorphism $P\rightarrow M$ with $P$ in $R\proj$. 
We inductively set $\Omega_{R}^{i+1}(M) := \Omega_{R}(\Omega_{R}^{i}(M))$ for each $i \geq 0$ with $\Omega_{R}^{0}(M) = M$.
Now, consider the composition, denoted by $G$, of the fully faithful functor $R\mod \rightarrow \D^{\rm b}(R\mbox{-}\mathrm{mod})$ and the canonical triangle functor $\D^{\rm b}(R\mbox{-}\mathrm{mod}) \rightarrow \D_{\rm sg}(R)$, where the first functor sends a finitely generated $R$-module $M$ to the stalk complex $M$ concentrated in degree $0$.
Since $G(R)=0$, 
there uniquely exists a canonical functor $F:R\smod$ $\rightarrow$ $\D_{\rm sg}(R)$ making the diagram 
\[\xymatrix@=20pt{
R\mbox{-}\mathrm{mod} \ar[r] \ar[d]_-{G}  & R\mbox{-}\underline{\mathrm{mod}}  \ar[ld]^-{F} \\
 \D_{\rm sg}(R) &
}\]
commute, where the horizontal functor $R\mod \rightarrow R\smod$ is the canonical functor.
For each $M \in R\mod$, we denote by $M$ the object $F(M)$  when no confusion occurs.
It was shown by Buchweitz \cite[Lemma 2.2.2]{Buch86} that $\Omega_{R}(M) \cong M[-1]$ in $\D_{\rm sg}(R)$.

%%%%%%%%%%%%%%%%%%%%%%%%%%%%%%%%%%%%%%%%%%%%%%%%%%
\subsection{Tate cohomology rings}
%%%%%%%%%%%%%%%%%%%%%%%%%%%%%%%%%%%%%%%%%%%%%%%%%%
This subsection is devoted to recalling some facts on Tate cohomology from \cite{Buch86,Wang_2021}.
Let us begin with the definition of Tate cohomology groups.

%
%%%%%%%%%%%%%%%      statement     %%%%%%%%%%%%%%%%%%%
%
\begin{defi}[{\cite[Definition 6.1.1]{Buch86} and \cite[page 11]{Wang_2021}}] \label{def_1} {\rm 
Let $i$ be an integer.
\begin{enumerate}
\item  Let $R$ be a left Noetherian ring and $M$ and $N$ two finitely generated $R$-modules. 
We define the {\it $i$-th Tate cohomology group of  $M$ with coefficients in  $N$} by
$\hExt_{R}^{i}(M, N) := \Hom_{\D_{\rm sg}(R)}(M, N[i])$.

\item  Let $\L$ be a (two-sided) Noetherian algebra over a commutative ring such that the enveloping algebra $\Le$ is Noetherian.
The {\it $i$-th Tate-Hochschild cohomology group} of $\L$ is defined by $\hHH^{i}(\Lambda):=\hExt_{\Le}^{i}(\Lambda, \Lambda)$.

\end{enumerate} 
}\end{defi}
%%%%%%%%%%%%%%%%%%%%%%%%%%%%%%%%%%%%%%%%%%%%%%%%%%%%%%

%
%%%%%%%%%%%%%%%      statement     %%%%%%%%%%%%%%%%%%%
% 
\begin{rem} \label{rem_3} 
{\rm  
Tate cohomology groups were originally introduced by Buchweitz \cite{Buch86} within a framework over Gorenstein rings, and the terminology \lq Tate cohomology\rq\  comes from his observation in \cite[Section 8.2]{Buch86} that $\hExt_{\mathbb{Z}G}^{*}(\mathbb{Z}, N) \cong \hHh^{*}(G, N)$, where $\hHh^{*}(G, N)$ denotes the original Tate cohomology group of a finite group $G$ with coefficients in a finitely generated  $\mathbb{Z}G$-module $N$. 
}\end{rem}
%%%%%%%%%%%%%%%%%%%%%%%%%%%%%%%%%%%%%%%%%%%%%%%

For each finitely generated module $M$ over a left Noetherian ring $R$, one defines a graded ring structure on the Tate cohomology $\hExt_{R}^{\bullet}(M, M):= \bigoplus_{i \in \mathbb{Z}} \hExt_{R}^{i}(M, M)$ by using the {\it Yoneda product}
\begin{align*}
     \hExt_{R}^{i}(M, M) \times \hExt_{R}^{j}(M, M) \rightarrow \hExt_{R}^{i+j}(M, M); \quad (\alpha,  \beta) \mapsto \alpha[j] \circ \beta.
\end{align*}
Such a graded ring $\hExt_{R}^{\bullet}(M, M)$ is called the {\it Tate cohomology ring} of $M$, and, for an algebra $\L$ as in Definition \ref{def_1}, we call the Tate cohomology ring $\hExt_{\Le}^{\bullet}(\Lambda, \Lambda)$ the {\it Tate-Hochschild cohomology ring} of $\L$ and denote it by $\hHH^{\bullet}(\Lambda)$.
It was proved by Wang \cite[Proposition 4.7]{Wang_2021} 
that $\hHH^{\bullet}(\Lambda)$ is graded commutative when the ground ring of $\L$ is a field.

Observe that a finitely generated $R$-module $M$ has finite projective dimension  if and only if $M \cong 0$ in $\D_{\rm sg}(R)$  if and only if  $\hExt_{R}^{i}(M, M) = 0$ for all $i \in \mathbb{Z}$ (cf. \cite[Section 1]{Buch86}). 
In this case, the Tate cohomology ring of $M$ is the zero ring.
For this reason, we  are mainly interested in the modules of infinite projective dimension.

%%%%%%%%%%%%%%%%%%%%%%%%%%%%%%%%%%%%
%           Section ↑
%%%%%%%%%%%%%%%%%%%%%%%%%%%%%%%%%%%%

%%%%%%%%%%%%%%%%%%%%%%%%%%%%%%%%%%%%
%           Section ↓
%%%%%%%%%%%%%%%%%%%%%%%%%%%%%%%%%%%%
\section{Main result} 
%%%%%%%%%%%%%%%%%%%%%%%%%%%%%%%%%%%%
In this section, we will prove our main result.
Let us first recall some basic facts on syzygies. 
Let $R$ be a left artin ring and $M$ a finitely generated $R$-module.
Then there exists a projective cover $p_M:P_M \rightarrow M$ with $P_M$ in $R\proj$, and hence the kernel $\Ker p_M$, called the {\it syzygy} of $M$, is uniquely determined up to isomorphism. 
We put, by abuse of notation, $\Omega_{R}(M):=\Ker p_M$ and define inductively $\Omega_{R}^{i+1}(M)=\Omega_{R}(\Omega_{R}^{i}(M))$ for $i \geq 0$ with  $\Omega_{R}^{0}(M)=M$.
Taking syzygies defines a well-defined endfunctor on $R\smod$, which agrees with the loop space functor on $R\smod$.
Throughout this section, the symbol $\Omega_R$ means taking syzygies unless otherwise stated.

Let $\L$ be an artin algebra, that is, an algebra over a commutative artin ring $S$ such that $\L$ is finitely generated over $S$.
If $\L$ is an artin algebra, then so is the enveloping algebra $\Le$.

Following K{\"u}pper \cite[Definition 2.3]{kupper2010two}, we now define eventually periodic algebras by using a concept of eventual periodicity for modules.

%
%%%%%%%%%%%%%%%      statement     %%%%%%%%%%%%%%%%%%%
%
\begin{defi} \label{def_4} {\rm 
\begin{enumerate}

\item 
Let $R$ be a left artin ring.
We call a finitely generated $R$-module $M$ {\it periodic} if $\Omega_{R}^{p}(M) \cong M$ as $R$-modules for some $p >0$. 
The least such $p$ is called the {\it period} of $M$.
We say that $M$ is {\it eventually periodic} if  $\Omega_{R}^{n}(M)$ is non-zero and periodic for some $n \geq 0$.

\item 
An artin algebra $\L$ is said to be {\it periodic} $($resp.\ {\it eventually periodic$)$} if the $\Le$-module $\L$ is periodic $($resp.\ eventually periodic$)$.

\end{enumerate} 
}\end{defi}
%%%%%%%%%%%%%%%%%%%%%%%%%%%%%%%%%%%%%%%%%%%%%%%%%%%%%%

%
%%%%%%%%%%%%%%%      statement     %%%%%%%%%%%%%%%%%%%
% 
\begin{rem} \label{rem_1} 
{\rm  
Our definition of eventually periodic modules is slightly stronger than the ordinary one (see \cite[page 170]{Bergh_2006} for example).
More precisely, eventually periodic modules in our sense are those of infinite projective dimension in the sense ever before. 
}\end{rem}
%%%%%%%%%%%%%%%%%%%%%%%%%%%%%%%%%%%%%%%%%%%%%%%

Let $M$ be an eventually periodic module whose $n$-th syzygy is periodic.  
Then it is easy to see that the $(n+i)$-th syzygy of $M$ with $i \geq 1$ is also periodic and has the same period as the $n$-th syzygy of $M$.
As a result, all the periodic syzygies of $M$ have the same period.

Recall that a finite dimensional $k$-algebra $\L$ is called {\it self-injective} if $\L$ is injective as a one-sided module over itself (see \cite{SkowYama11}); 
and that the algebra $\L$ is called {\it Gorenstein} if it has finite injective dimension on each side (see \cite{Buch86}). 
Self-injective algebras and algebras of finite global dimension are examples of Gorenstein algebras.

For the case of finite dimensional $k$-algebras, many authors have studied periodic algebras (see \cite{ErdSko08,ErdSko19}). 
In particular, we know from \cite[Proposition IV.11.18]{SkowYama11} that finite dimensional periodic algebras are self-injective and hence Gorenstein. 
On the other hand, finite dimensional eventually periodic algebras can be found in several papers. 
For example, Dotsenko, Gélinas and Tamaroff showed in the proof of \cite[Corollary 6.4]{DGT19} that  monomial Gorenstein algebras $\L$ having infinite projective dimension over $\Le$ are eventually periodic; and the author \cite[Sections 3.1 and 4]{Usui21_No02} provided examples of finite dimensional eventually periodic algebras (that are not periodic), and he also confirmed that there is an eventually periodic algebra that is not Gorenstein.
As will be seen in Corollary \ref{claim_30}, finite dimensional connected Nakayama algebras $\L$ with $\projdim_{\Le}\L=\infty$ are eventually periodic (when $k$ is algebraically closed). 
Also, K{\"u}pper \cite[Lemma 2.7]{kupper2010two} has given a necessary and sufficient condition for a monomial algebra to be eventually periodic, using a minimal set of monomial relations.

Let us prepare for the following lemma, which is probably well known to some experts. 
We include a proof for the benefit of the reader.
%Recall that $\Omega_{i}(X_\bullet) = \Coker d_{i+1}^{X}$ for a complex $X_\bullet$ and an integer $i$.

%
%%%%%%%%%%%%%%%      statement     %%%%%%%%%%%%%%%%%%%
%
\begin{lemma} \label{claim_12} 
Let $R$ be a left Noetherian ring, and let $\alpha: X_\bullet \rightarrow Y_\bullet$ be a morphism in $\K^{\rm -, b}(R\proj)$.
If $\Cone(\alpha)_\bullet$ is perfect, then there exists an integer $l$ such that $\Omega_{i}(X_\bullet) \cong \Omega_{i}(Y_\bullet)$ in $R\smod$ for all $i \geq l$.
\end{lemma}

\begin{proof}
Assume that $\Cone(\alpha)_\bullet$ is perfect.
Since $\Cone(\alpha)_\bullet$ has bounded homology, there exists an integer $r$ such that the augmented complex 
\begin{align} \label{eq_3}
    \Cone(\alpha)_{\geq r} \xrightarrow{\pi_r} \Omega_{r}(\Cone(\alpha)_\bullet)
\end{align}
is contractible, where $\Cone(\alpha)_{\geq r}$ is the stupid truncation of $\Cone(\alpha)_\bullet$. 
Note that $\Omega_{r}(\Cone(\alpha)_\bullet)$ is in $R\proj$.
Let $\iota_{Y_r}$ and $\iota_{X_{r-1}}$ be the canonical inclusions $Y_r \hookrightarrow \Cone(\alpha)_{r}$ and $X_{r-1} \hookrightarrow \Cone(\alpha)_{r}$, respectively.
Then one gets the chain map $\alpha^{\prime}: X_{\geq r-1}\rightarrow Y^\prime_\bullet$ given as follows:
\[\xymatrix{
X_{\geq r-1} = \cdots \ar[r] \ar@<-3ex>[d]_-{\alpha^\prime} & 
X_{r+2} \ar[r]^-{d_{r+2}^{X}} \ar[d]^-{\alpha_{r+2}} & 
X_{r+1} \ar[r]^-{d_{r+1}^{X}} \ar[d]^-{\alpha_{r+1}} &
X_{r} \ar[r]^-{d_{r}^{X}} \ar[d]^-{\alpha_{r}} & X_{r-1} \ar[r] \ar[d]^-{\pi_{r}\circ\iota_{X_{r-1}}}&
0 \ar[r] \ar[d] &
\cdots\\
Y^\prime_\bullet  = \cdots \ar[r] &
Y_{r+2} \ar[r]_-{d_{r+2}^{Y}} &
Y_{r+1} \ar[r]_-{d_{r+1}^{Y}} &
Y_{r} \ar[r]_-{\pi_{r}\circ\iota_{Y_r}} &
\Omega_{r}(\Cone(\alpha)) \ar[r] &
0 \ar[r]  &
\cdots 
}\]  
Since  $\Cone(\alpha^{\prime})$ is nothing but the contractible complex (\ref{eq_3}), the chain map  $\alpha^{\prime}: X_{\geq r-1}\rightarrow Y^\prime_\bullet$ is a homotopy equivalence.
Hence there exists a chain map $\beta^\prime: Y^\prime \rightarrow X_{\geq r-1}$ such that   $\beta^\prime \alpha^\prime \sim \id_{X_{\geq r-1}}$ and $\alpha^\prime  \beta^\prime \sim \id_{Y^\prime}$.
Since $X_\bullet$ and $Y_\bullet$ have bounded homology, there exists an integer $l > r$ such that $\Omega_{i}(X_\bullet)=\Omega_{i}(X_{\geq r-1}) \cong \Omega_{i}(Y^\prime_\bullet)=\Omega_{i}(Y_\bullet)$ in $R\smod$ for each $i \geq l$, where the isomorphism is induced by the homotopy equivalences $\alpha^\prime$ and $\beta^\prime$.
\end{proof}
%%%%%%%%%%%%%%%%%%%%%%%%%%%%%%%%%%%%%%%%%%%%%%%%%%%%%%

We now characterize the existence of invertible homogeneous elements in Tate cohomology rings in terms of projective resolutions.
Notice that $\Omega_R$ denotes the loop space functor in the following theorem.

%
%%%%%%%%%%%%%%%      statement     %%%%%%%%%%%%%%%%%%%
%
\begin{theo} \label{claim_13}  
Let $R$ be a left Noetherian ring, and let $p > 0$ be an integer. 
Then the following conditions are equivalent for a finitely generated $R$-module $M$.
\begin{enumerate}
    \item There exists an integer $n \geq 0$ such that $\Omega_R^{n+p}(M) \cong \Omega_R^{n}(M)$ in $R\smod$.
    \item The Tate cohomology ring  $\hExt_{R}^{\bullet}(M, M)$  has an invertible homogeneous element of degree $p$.  
\end{enumerate}
\end{theo}

\begin{proof}
We first prove that $(1)$ implies $(2)$.  
Let $f \in \underline{\Hom}_{R}(\Omega_{R}^{n+p}(M), \Omega_{R}^{n}(M))$ be an isomorphism, and recall that  the canonical functor $R\smod \rightarrow \D_{\rm sg}(R)$ is denoted by $F$. 
Consider the following homogeneous element
\begin{align*}
    x := F(f) \in \Hom_{\D_{\rm sg}(R)}(\Omega_{R}^{n+p}(M), \Omega_{R}^{n}(M)) &\cong \Hom_{\D_{\rm sg}(R)}(M[-p-n], M[-n])\\ &\cong \Hom_{\D_{\rm sg}(R)}(M, M[p]) \\ &= \hExt_{R}^{p}(M, M).
\end{align*}
Then $y:= F(f^{-1}) \in \Hom_{\D_{\rm sg}(R)}(\Omega_{R}^{n}(M), \Omega_{R}^{n+p}(M)) \cong \hExt_{R}^{-p}(M, M)$ is  the inverse of $x$ in the graded ring $\hExt_{R}^{\bullet}(M, M)$.

Conversely, suppose that there exists an isomorphism
\[
x \in \Hom_{\D_{\rm sg}(R)}(\Omega_{R}^{n+p}(M), \Omega_{R}^{n}(M)) \cong  \hExt_{R}^{p}(M, M)
\]
for some  $n \geq 0$.
Then it is represented by some fraction $\Omega_{R}^{n+p}(M) \xleftarrow{t} W_{\bullet} \xrightarrow{\alpha} \Omega_{R}^{n}(M)$ of morphisms in $\D^{\rm b}(R\mod)$ such that both $\Cone(t)_\bullet$ and $\Cone(\alpha)_\bullet$ are perfect.
Let $P_{\bullet}\rightarrow M$ be a minimal projective resolution of the $R$-module $M$. 
Then the augmented complex $P_{\geq n+p}[-n-p]\rightarrow \Omega_{R}^{n+p}(M)$ is a minimal projective resolution of the $R$-module $\Omega_{R}^{n+p}(M)$.
Let $P_\bullet^W \rightarrow W_\bullet$ be a quasi-isomorphism with $P_\bullet^W$ in $\K^{\rm -, b}(R\proj)$.
Then one has isomorphisms
\begin{align*}
    \Hom_{\D^{\rm b}(R\mod)}(W_\bullet, \Omega_{R}^{n+p}(M)) &\cong \Hom_{\D^{\rm b}(R\mod)}(W_\bullet, P_{\geq n+p}[-n-p])\\ &\cong \Hom_{\K^{\rm -, b}(R\proj)}(P_\bullet^W, P_{\geq n+p}[-n-p]).
\end{align*}
Similarly, one gets \[\Hom_{\D^{\rm b}(R\mod)}(W_\bullet, \Omega_{R}^{n}(M)) \cong \Hom_{\K^{\rm -, b}(R\proj)}(P_\bullet^W, P_{\geq n}[-n]).\]
Using the two morphisms in $\K^{\rm -, b}(R\proj)$ corresponding to $t$ and $\alpha$ via the above isomorphisms, we see from Lemma \ref{claim_12} that there exists a sufficiently large integer $l$ such that 
\[\Omega_{R}^{n+l}(M) = \Omega_{l}(P_{\geq n}[-n]) \cong \Omega_{l}(P_\bullet^W) \cong \Omega_{l}(P_{\geq n+p}[-n-p])= \Omega_{R}^{n+p+l}(M) \]
in $R\smod$. 
This completes the proof.
\end{proof}
%%%%%%%%%%%%%%%%%%%%%%%%%%%%%%%%%%%%%%%%%%%%%%%%%%%%%%

%%%%%%%%%%%%%%%%%%%%%%%%%%%%%%
%        remark     ↓
%%%%%%%%%%%%%%%%%%%%%%%%%%%%%%
\begin{rem}{\rm 
For a left Noetherian ring $R$, it follows from \cite[Corollary 3.9]{Beligiannis_2000} applied to $\mathcal{C}=R\mod$ and $\mathcal{P} = R\proj$ that $\D_{\rm sg}(R)$ is triangle equivalent to the stabilization $\mathcal{S}(R\mod)$ of $R\smod$ (cf.\ \cite[Section 6.5]{Buch86}).
Thus if $M$ is a finitely generated $R$-module and $p$ is an integer, then we have
\begin{align*}
%    \hExt_{R}^{p}(M, M) \cong \Hom_{\mathcal{S}(R\mod)}((M,0), (M, p)) = \lim_{\substack{\longrightarrow \\ k \in I_{0, p}}}\, \underline{\Hom}_{R}(\Omega_R^{k}(M), \Omega_R^{k-p}(M)),
    \hExt_{R}^{p}(M, M) \cong \Hom_{\mathcal{S}(R\mod)}((M,0), (M, p)) = \lim_{\substack{ k \rightarrow \infty}}\, \underline{\Hom}_{R}(\Omega_R^{k}(M), \Omega_R^{k-p}(M)).
\end{align*}
%where we put $I_{n, m} := \{\, k \in \mathbb{Z}\,|\,k \geq n, k \geq m \,\}$ for any $n$ and $m \in \mathbb{Z}$, and the maps of the directed inductive system are given by the loop space functor $\Omega_{R}$. 
Note that the directed colimit is called the {\it $p$-th Tate-Vogel cohomology group} of $M$. 
The above isomorphism enables us to show the implication from (2) to (1) in the proof of Theorem \ref{claim_13} without Lemma \ref{claim_12}  (use \cite[Corollary 3.3 (3)]{Beligiannis_2000} for example).
}\end{rem}
%%%%%%%%%%%%%%%%%%%%%%%%%%%%%%
%        remark     ↑
%%%%%%%%%%%%%%%%%%%%%%%%%%%%%%

The main result of this paper is obtained from the above theorem and gives a necessary and sufficient condition for a  module to be eventually periodic.

%%%%%%%%%%%%%%%%%%%%%%%%%%%%%%
%        statement  ↓
%%%%%%%%%%%%%%%%%%%%%%%%%%%%%%
\begin{cor} \label{claim_13_1}  
Let $R$ be a left artin ring. 
Then the following conditions are equivalent for a finitely generated $R$-module $M$.
\begin{enumerate}
    \item $M$ is eventually periodic.
    \item The projective dimension of $M$ is infinite, and the Tate cohomology ring of $M$  has an invertible homogeneous element of positive degree.  
\end{enumerate}
In this case, there exists an invertible homogeneous element in the Tate cohomology ring of $M$ whose degree equals the period of some periodic syzygy $\Omega_R^{n}(M)$ with $n \geq 0$.
\end{cor}

\begin{proof}
To prove that (2) implies (1), we assume that there exists a (non-zero) invertible homogeneous element of degree $p > 0$ in $\hExt_{R}^{\bullet}(M, M)$.
Then it follows from Theorem \ref{claim_13} that $\Omega_{R}^{n+p}(M) \cong \Omega_{R}^{n}(M)$ in $R\smod$ for some $n \geq 0$, so that there exist two finitely generated projective $R$-modules $P$ and $Q$ such that  $\Omega_{R}^{n+p}(M) \oplus P \cong \Omega_{R}^{n}(M) \oplus Q$ in   $R\mod$.  
Taking their syzygies, we have an isomorphism $\Omega_{R}^{n+p+1}(M)\cong \Omega_{R}^{n+1}(M)$ of $R$-modules.

Now, suppose that $M$ is an eventually periodic $R$-module whose $n$-th syzygy is periodic of period $p$.  
Then we have $\projdim_R M = \infty$ and, by Theorem \ref{claim_13}, the graded ring $\hExt_{R}^{\bullet}(M, M)$ has an invertible homogeneous element of degree $p$. 
This finishes the proof.
\end{proof}
%%%%%%%%%%%%%%%%%%%%%%%%%%%%%%
%        statement   ↑  
%%%%%%%%%%%%%%%%%%%%%%%%%%%%%%

%%%%%%%%%%%%%%%%%%%%%%%%%%%%%%
%        remark     ↓
%%%%%%%%%%%%%%%%%%%%%%%%%%%%%%
\begin{rem} {\rm 
For a commutative Noetherian local ring $R$, one can define eventually periodic $R$-modules by using minimal free resolutions as in Definition \ref{def_4} (see \cite[Section 8]{Avramov-Martsinkovsky_2002} for minimality of free resolutions).
Then Theorem \ref{claim_13} can be used to characterize eventually periodic $R$-modules as in Corollary \ref{claim_13_1}.
We leave to the reader to state and show the analogous result.
}\end{rem}
%%%%%%%%%%%%%%%%%%%%%%%%%%%%%%
%        remark     ↑
%%%%%%%%%%%%%%%%%%%%%%%%%%%%%%

We end this section by giving a result on Tate-Hochschild cohomology rings.
Let $\L$ be an artin algebra.
Then the above corollary applied to ${}_{R}M = {}_{\Le}\L$ yields the following result, which generalizes \cite[Corollary 6.4]{DGT19}, \cite[Corollary 3.4]{Usui21} and \cite[Theorem 3.5]{Usui21_No02} (in a proper sense).

%
%%%%%%%%%%%%%%%      statement     %%%%%%%%%%%%%%%%%%%
%
\begin{cor} \label{claim_14} 
We have the following statements.
\begin{enumerate}
\item The following conditions are equivalent for an artin algebra $\L$.
\begin{enumerate}
    \item $\L$ is eventually periodic.
    \item The projective dimension of the $\Le$-module $\L$ is infinite, and the Tate-Hochschild cohomology ring of $\L$ has an invertible homogeneous element of positive degree.  
\end{enumerate}   
\item Let $\L$ be a finite dimensional eventually periodic $k$-algebra. Then there exists an isomorphism of graded $k$-algebras \[\hHH^{\bullet}(\L) \cong \hHH^{\geq 0}(\L)[\chi^{-1}],\] where we set  $\hHH^{\geq 0}(\Lambda) :=\bigoplus_{i\geq 0}\hHH^{i}(\Lambda)$, and the degree of the invertible homogeneous element $\chi$ is equal to  the period of some periodic syzygy $\Omega_{\Le}^{n}(\L)$ with $n \geq 0$.
\end{enumerate}  
\end{cor}

\begin{proof}
This follows from Corollary \ref{claim_13_1} and the fact that $\hHH^{\bullet}(\L)$ is a graded commutative ring for a finite dimensional $k$-algebra $\L$.
\end{proof}
%%%%%%%%%%%%%%%%%%%%%%%%%%%%%%%%%%%%%%%%%%%%%%%%%%%%%%

%%%%%%%%%%%%%%%%%%%%%%%%%%%%%%%%%%%%
%           Section ↑
%%%%%%%%%%%%%%%%%%%%%%%%%%%%%%%%%%%%

%%%%%%%%%%%%%%%%%%%%%%%%%%%%%%%%%%%%
%           Section ↓
%%%%%%%%%%%%%%%%%%%%%%%%%%%%%%%%%%%%
\section{Applications} 

There are two aims of this section. 
The first is to prove that, for the class of finite dimensional $k$-algebras, the property of being eventually periodic is preserved under singular equivalence of Morita type with level. 
The second is to provide  a necessary and sufficient condition for a finite dimensional connected Nakayama $k$-algebra to be eventually periodic when the ground field $k$ is algebraically closed.

%%%%%%%%%%%%%%%%%%%%%%%%%%%%%%%%%%%%
\subsection{Eventually periodic algebras and singular equivalences of Morita type with level} 
%%%%%%%%%%%%%%%%%%%%%%%%%%%%%%%%%%%%

Throughout this subsection, unless otherwise stated, $\Omega_R$ denotes the loop space functor on $R\smod$.

This subsection is devoted to showing that singular equivalences of Morita type with level preserve the eventual periodicity of finite dimensional $k$-algebras.  
Let us start with the notion of singular equivalences, which was introduced by Chen \cite{X-WChen_2014}. 
We say that  two Noetherian $k$-algebras $\L$ and $\G$ are {\it singularly equivalent} if there exists a triangle equivalence between $\D_{\rm sg}(\L)$ and $ \D_{\rm sg}(\G)$.
Such a triangle equivalence is called a {\it singular equivalence}. 
On the other hand, Wang \cite{Wang_2015} introduced the notion of {\it singular equivalences of Morita type with level} and observed that such equivalences yield singular equivalences between two algebras and between their enveloping algebras.

%
%%%%%%%%%%%%%%%      statement     %%%%%%%%%%%%%%%%%%%
%
\begin{defi}[{\cite[Definition 2.1]{Wang_2015}}]  \label{def_3} 
{\rm   
Let $\L$ and $\G$ be two Noetherian $k$-algebras such that the enveloping algebras $\Le$ and $\Ge$ are Noetherian.
Let $l$ be a non-negative integer and ${}_{\L}M_{\G}$ and ${}_{\G}N_{\L}$ two bimodules.
We say that $({}_{\L}M_{\G},\, {}_{\G}N_{\L})$ {\it defines a singular equivalence of Morita type with level $l$} (and that $\L$ and $\G$ are {\it singularly equivalent of Morita type with level} $l$) if the following conditions are satisfied.
\begin{enumerate}
    \item The one-sided modules ${}_{\L}M, M_{\G}, {}_{\G}N$ and $N_{\L}$ are finitely generated and projective.
    \item  There exist isomorphisms $M \otimes_{\G} N \cong \Omega_{\Le}^{l}(\L)$ and $N \otimes_{\L} M \cong \Omega_{\Ge}^{l}(\G)$ in $\Le\smod$ and $\Ge\smod$, respectively.
\end{enumerate}
}\end{defi}
%%%%%%%%%%%%%%%%%%%%%%%%%%%%%%%%%%%%%%%%%%%%%%%%%%%%%%

%
%%%%%%%%%%%%%%%      statement     %%%%%%%%%%%%%%%%%%%
% 
\begin{rem} \label{rem_2}
{\rm
As mentioned in \cite[Remark 2.1]{Wang_2015} and \cite[Remark 6.1]{Wang_2019}, one easily sees that any pair $({}_{\L}M_{\G},\, {}_{\G}N_{\L})$ defining a singular equivalence of Morita type with level $l$ induces two well-defined singular equivalences
\[N \otimes_{\L}-: \D_{\rm sg}(\L) \rightarrow \D_{\rm sg}(\G)\]
and 
\[\Phi:=(N \otimes_{\L}-\otimes_{\L} M) [l] : \D_{\rm sg}(\Le) \rightarrow \D_{\rm sg}(\Ge)\]
with quasi-inverses
\[(M \otimes_{\G}-)[l]: \D_{\rm sg}(\G) \rightarrow \D_{\rm sg}(\L)\]
and 
\[ (M \otimes_{\G}-\otimes_{\G} N) [l] : \D_{\rm sg}(\Ge) \rightarrow \D_{\rm sg}(\Le),\]
respectively.
Note that $\Phi(\L) \cong \G$ in $\D_{\rm sg}(\Ge)$.
It was proved by Wang \cite[Theorem 1.1]{Wang_2019} that the singular equivalence $\Phi: \D_{\rm sg}(\Le) \rightarrow \D_{\rm sg}(\Ge)$ induces an isomorphism $\hHH^{\bullet}(\L) \cong \hHH^{\bullet}(\G)$ of graded algebras.
}\end{rem}
%%%%%%%%%%%%%%%%%%%%%%%%%%%%%%%%%%%%%%%%%%%%%%%

%
%%%%%%%%%%%%%%%      statement     %%%%%%%%%%%%%%%%%%%
% 
\begin{rem}
{\rm
Let $\L, \G$ and $\Sigma$ be finite dimensional $k$-algebras, and let   $\Omega_{\L\mbox{-}\G}$ denote $\Omega_{\L \otimes_k \G^{\rm op}}$.
Write $\L \sim \G$ if $\L$ and $\G$ are singularly equivalent of Morita type with level.
Then if $\L \sim \G$ and $\G \sim \Sigma$, then $\L \sim \Sigma$.
%If $\L$ and $\G$ are singularly equivalent of Morita type with level, and if $\G$ and $\Sigma$ are singularly equivalent of Morita type with level, then so are $\L$ and $\Sigma$. 
This follows from the observation that if $\left({}_{\L}M_{\G},\, {}_{\G}N_{\L}\right)$ and  $\left({}_{\G}X_{\Sigma},\, {}_{\Sigma}Y_{\G}\right)$ define singular equivalences of Morita type with level $l$ and $w$, respectively, then  $\left({}_{\L}M \otimes_{\G} X_{\Sigma},\, {}_{\Sigma}Y \otimes_{\G} N_{\L}\right)$ defines a singular equivalence of Morita type with level $l + w$.
Here, we use isomorphisms $\Omega_{\L\mbox{-}\G}(M) \cong M \otimes_\G \Omega_{\Le}(\L)$ in $(\L \otimes_k \G^{\rm op})\smod$ and $\Omega_{\Le}(M\otimes_\G N) \cong   \Omega_{\L\mbox{-}\G}(M) \otimes_{\G} N \cong  M \otimes_\G  \Omega_{\G\mbox{-}\L}(N)$ in $\Le\smod$. 
Thus one concludes that $ \sim $ becomes an equivalence relation on the class of finite dimensional $k$-algebras.
%Thus one concludes that singular equivalences of Morita type with level induce an equivalence relation on the class of finite dimensional $k$-algebras.
}\end{rem}
%%%%%%%%%%%%%%%%%%%%%%%%%%%%%%%%%%%%%%%%%%%%%%%

Before proving the main result of this subsection, we prepare for two statements, which will be used in the next subsection.
The first shows that any two algebras are singularly equivalent of Morita type with level whenever they have finite projective dimension as bimodules.

%
%%%%%%%%%%%%%%%      statement     %%%%%%%%%%%%%%%%%%%
%
\begin{lemma} \label{claim_31} 
Let $\L$ and $\G$ be as in Definition $\ref{def_3}$.
If $\projdim_{\Le}\L$ and  $\projdim_{\Ge}\G$ are both finite, then $\L$ and $\G$ are singularly equivalent of Morita type with level.
\end{lemma}

\begin{proof}
Assume that $\projdim_{\Le}\L < \infty$ and  $\projdim_{\Ge}\G < \infty$, and let \[l :=  \mathrm{max}\left\{\,\projdim_{\Le}\L,\, \projdim_{\Ge}\G\, \right\}.\] 
Then 
$0 \otimes_\L 0 \cong 0 \cong \Omega_{\Le}^{l}(\L)$ in $\Le\smod$, and similarly $0 \otimes_\G 0  \cong \Omega_{\Ge}^{l}(\G)$
in $\Ge\smod$. 
This implies that $({}_{\L}0_{\G},\, {}_{\G}0_{\L})$ defines a singular equivalence of Morita type with level $l$.
\end{proof}
%%%%%%%%%%%%%%%%%%%%%%%%%%%%%%%%%%%%%%%%%%%%%%%%%%%%%%

It follows from Definition \ref{def_3} that singular equivalences of Morita type with level $0$ are nothing but {\it stable equivalences of Morita type} in the sense of Brou\'{e} \cite{Broue_1994} (cf. \cite[Remark 4.3]{Dalezios_2021}).
The next observation tells us that any two self-injective algebras are stably equivalent of Morita type whenever they are singularly equivalent of Morita type with level (compare \cite[Proposition 3.7]{Skart_2016}).

%
%%%%%%%%%%%%%%%      statement     %%%%%%%%%%%%%%%%%%%
%
\begin{lemma} \label{claim_32} 
For two finite dimensional self-injective $k$-algebras $\L$ and $\G$, the following conditions are equivalent.
\begin{enumerate}
    \item $\L$ and $\G$ are singularly equivalent of Morita type with level.
    \item $\L$ and $\G$ are stably equivalent of Morita type.
\end{enumerate}
\end{lemma}

\begin{proof}
It suffices to prove that (1) implies (2).
Suppose that $({}_{\L}M_{\G},\, {}_{\G}N_{\L})$ defines a singular equivalence of Morita type with level $l$.
Recall that if $\L$ and $\G$ are finite dimensional self-injective $k$-algebras, then so are $\G^{\rm op}$ and $\L \otimes \G^{\rm op}$.
Then it is easy to see that $\Omega_{\Le}^{-1}(M \otimes_{\G} N) \cong \Omega_{\L\mbox{-}\G}^{-1}(M) \otimes_{\G} N \cong M \otimes_{\G}  \Omega_{\G\mbox{-}\L}^{-1}(N)$ 
in $\Le\smod$, where $\Omega_{\Le}^{-1}(M \otimes_{\G} N)$ denotes the cokernel of an injective envelope of the $\Le$-module $M \otimes_{\G} N$.
Hence one obtains $\Omega_{\L\mbox{-}\G}^{-l}(M) \otimes_{\G} N \cong \Omega_{\Le}^{-l}(\Omega_{\Le}^{l}(\L)) \cong \L$
in $\Le\smod$ and similarly $N \otimes_{\L} \Omega_{\L\mbox{-}\G}^{-l}(M) \cong \G$ in $\Ge\smod$.
Thus we conclude that $( \Omega_{\L\mbox{-}\G}^{-l}(M),\, N)$ defines a singular equivalence of Morita type with level $0$ between $\L$ and $\G$.
\end{proof}
%%%%%%%%%%%%%%%%%%%%%%%%%%%%%%%%%%%%%%%%%%%%%%%%%%%%%%

Now, we prove the main result of this subsection, which is an easy consequence of Corollary \ref{claim_14} and Wang's result \cite{Wang_2019}.

%
%%%%%%%%%%%%%%%      statement     %%%%%%%%%%%%%%%%%%%
%
\begin{theo} \label{claim_15}  
Assume that two finite dimensional $k$-algebras $\L$ and $\G$ are singularly equivalent of Morita type with level. 
If $\L$ is eventually periodic, then so is $\G$. 
In particular, the periods of their  periodic syzygies coincide.
\end{theo} 

\begin{proof}  
Suppose that $\L$ is eventually periodic.
By \cite[Theorem 1.1]{Wang_2019} and Corollary \ref{claim_14}, there exist isomorphisms of graded $k$-algebras 
\begin{align} \label{eq_2}
    \hHH^{\bullet}(\G) \cong \hHH^{\bullet}(\L) \cong \hHH^{\geq 0}(\L)[\chi_{\L}^{-1}].
\end{align} 
Then it follows from Corollary \ref{claim_14} that $\G$ is also eventually periodic.

Now, let $p$ and $q$ be the periods of some periodic syzygies of the regular bimodules $\L$ and $\G$, respectively.
We claim that $p=q$.
It follows from the isomorphism (\ref{eq_2}) that $\hHH^{\bullet}(\G)$ has an invertible homogeneous element of degree $p$.
Then an argument as in the proof of the implication from (1) to (2) in Corollary \ref{claim_13_1} shows that there exists an isomorphism $\Omega_{\Ge}^{j+p}(\G) \cong \Omega_{\Ge}^{j}(\G)$ in $\Ge\mod$ for some $j \geq 0$, where $\Omega_{\Ge}$ denotes the syzygy over $\Ge$.
Since the periodic syzygy $\Omega_{\Ge}^{j}(\G)$ has period $q$, we obtain that $q$ divides $p$.
Similarly, the existence of an invertible homogeneous element of degree $q$ in $\hHH^{\bullet}(\L)$ implies that $p$ divides $q$.
Since $p$ and $q$ are positive, we conclude that $p=q$.
\end{proof}
%%%%%%%%%%%%%%%%%%%%%%%%%%%%%%%%%%%%%%%%%%%%%%%%%%%%%%

Recall that two finite dimensional $k$-algebras $\L$ and $\G$ are  {\it derived equivalent} if there exists a triangle equivalence between $\D^{\rm b}(\L\mod)$ and $\D^{\rm b}(\G\mod)$ (see \cite{Rick_1989_JournalOfTheLondon}).
It was proved by Wang \cite[Theorem 2.3]{Wang_2015} that any two derived equivalent finite dimensional $k$-algebras are singularly equivalent of Morita type with level.
Thus we obtain the following result, which generalizes a result of Erdmann and Skowro\'{n}ski  \cite[Theorem 2.9]{ErdSko08}.

%
%%%%%%%%%%%%%%%      statement     %%%%%%%%%%%%%%%%%%%
%
\begin{cor} \label{claim_17}  
Let $\L$ and $\G$ be two derived equivalent finite dimensional $k$-algebras. 
If $\L$ is eventually periodic, then so is $\G$. 
In particular, the periods of their  periodic syzygies coincide.
\end{cor} 
%%%%%%%%%%%%%%%%%%%%%%%%%%%%%%%%%%%%%%%%%%%%%%%%%%%%%%

We end this subsection with examples of eventually periodic algebras.
Note that one can find $\G$ and $\Sigma$ below in  \cite[Example 4.3 (2)]{X-WChen_2009} and \cite[Example 3.2 (1)]{Usui21_No02}, respectively.

%
%%%%%%%%%%%%%%%      statement     %%%%%%%%%%%%%%%%%%%
%
\begin{example}  \label{example_3}
{\rm
Let $\L$, $\G$ and $\Sigma$ be the $k$-algebras given by the following quivers with relations
\begin{align*}
    \xymatrix{1 \ar@(ul,dl)_-{\alpha}} \qquad \quad  \alpha^2 = 0,
\end{align*}
\begin{align*}
    \xymatrix{1 \ar@(ul,dl)_-{\alpha}\ar[r]^-{\beta} &2} \qquad \quad  \alpha^2 = 0 = \beta \alpha
\end{align*}
and 
\begin{align*}
    \xymatrix{1 \ar@<0.5ex>[r]^-{\alpha} &2\ar@<0.5ex>[l]^-{\beta}} \qquad \quad  \alpha \beta \alpha =0,
\end{align*}
respectively.
Then $\L$ is a self-injective Nakayama and hence a periodic algebra (see \cite[Lemma in Section 4.2]{ErdHolm99}). 
Note that $\L$ has period $1$ if the characteristic of $k$ is $2$ and $2$ otherwise.
The remaining algebras $\G$ and $\Sigma$ are not Gorenstein since $\injdim_{\G}P_1=\infty=\injdim_{\Sigma}P_1$, where $P_i$ denotes a projective cover of the simple module associated to the vertex $i$.
By \cite[Example 7.5]{Skart_2016}, the algebras $\L$ and $\G$ are singularly equivalent of Morita type with level $1$, so that $\G$ is eventually periodic by Theorem \ref{claim_15}.
Letting $\tau$ be the Auslander-Reiten translation, one obtains $\Sigma = \End_\G(T)^{\rm op}$, where $T$ is the APR-tiling $\G$-module $T[2] = \tau^{-1} P_2 \oplus P_1$.
This implies that $\G$ is derived equivalent to $\Sigma$, and hence $\Sigma$ is eventually periodic by Corollary \ref{claim_17}. 
On the other hand, since $\G$ and $\Sigma$ are monomial algebras, using a result of Bardzell \cite{Bard97}, we have that $\Omega_{\Ge}^{3}(\G)$ and $\Omega_{\Sigma^{\rm e}}^{2}(\Sigma)$ are the first periodic syzygies.
Thus Theorem \ref{claim_15} implies that the periodic bimodules $\L$, $\Omega_{\Ge}^{3}(\G)$ and $\Omega_{\Sigma^{\rm e}}^{2}(\Sigma)$ have the same period.
}\end{example} 
%%%%%%%%%%%%%%%%%%%%%%%%%%%%%%%%%%%%%%%%%%%%%%%%%%%%%%

%%%%%%%%%%%%%%%%%%%%%%%%%%%%%%%%%%%%
\subsection{Eventually periodic Nakayama algebras} 
%%%%%%%%%%%%%%%%%%%%%%%%%%%%%%%%%%%%
Throughout this subsection, we suppose that the  field $k$ is algebraically closed unless otherwise stated.
We denote by $\CN$ the class of finite dimensional connected Nakayama $k$-algebras (see \cite{ASS06,ARS95} for their definition).
The aim of this subsection is to decide which algebras from $\CN$ are eventually periodic. 
Let us first review some facts from \cite[Section V.3]{ASS06}.
Let $\L$ be in $\CN$ and $J(\L)$ its Jacobson radical.
The algebra $\L$ is Morita equivalent to a bound quiver algebra whose ordinary quiver is given by either
\[\xymatrix@C=7mm{
1\ar[r] & 2 \ar[r] & \cdots \ar[r]& e-1 \ar[r] & e
}\]
or 
\begin{center}
\begin{minipage}{0.05\columnwidth}
$Z_e :$
\end{minipage}
\begin{minipage}{0.5\columnwidth}
\xymatrix@C=4mm@R=4mm{
& 1\ar[r] & 2 \ar[rd] & \\
e \ar[ru] & & & 3 \ar[ld] \\ 
 & e-1 \ar[lu] & \cdots \ar[l] &
}
\end{minipage}
\end{center}
where $e \geq 1$.
Note that $\gldim \L < \infty$ for the first case.
Moreover, $\L$ is non-simple and self-injective if and only if it is Morita equivalent to the bound quiver algebra $k Z_{e} / R^{N}$ for some $e \geq 1$ and $N \geq 2$, where $R$ denotes the arrow ideal of the path algebra $k Z_e$.

Now, we show the following key lemma, which is a consequence of results due to Qin \cite{Qin_2021} and Shen \cite{Shen_2015,Shen_2020}.

%
%%%%%%%%%%%%%%%      statement     %%%%%%%%%%%%%%%%%%%
%
\begin{lemma} \label{claim_35}   
Let $\L$ be a finite dimensional connected Nakayama $k$-algebra.
If the global dimension of $\L$ is infinite, then there exists an idempotent $f$ of $\L$ satisfying the following conditions.
\begin{enumerate}
    \item[(i)] $f \L f$ is a finite dimensional  connected self-injective Nakayama $k$-algebra.
    \item[(ii)] $\L$ and $f\L f$ are singularly equivalent of Morita type with level.
\end{enumerate}
\end{lemma}

\begin{proof}
Assume that $\gldim \L = \infty$.
It follows from the proof of \cite[Proposition 3.8]{Shen_2015} that there exists an idempotent $f \in \L$ such that $f\L f$ is a  connected self-injective Nakayama $k$-algebra and such that $\L f$ is a projective $(f\L f)^{\rm op}$-module.  
Note that the exact functor $i_{\lambda}:=\L f \otimes_{f \L f}-: f\L f\mod \rightarrow \L\mod$ naturally induces a triangle functor, still denoted by $i_{\lambda}$, $D^{\rm b}(f \L f\mod) \rightarrow D^{\rm b}(\L\mod)$ preserving perfect complexes.
By \cite[Theorem 3.11]{Shen_2015}, the functor $i_{\lambda}:f \L f\mod \rightarrow \L\mod$ induces a singular equivalence between $f \L f$ and $\L$,  that is, the induced triangle functor $D_{\rm sg}(i_{\lambda}):D_{\rm sg}(f \L f) \rightarrow D_{\rm sg}(\L)$ is an equivalence.
Then \cite[Theorem 4.2]{Shen_2020} yields 
$\injdim_{\L}\frac{\L/ \L f \L}{J(\L/ \L f \L)} < \infty,$
so that one concludes from \cite[Theorem 4.2]{Qin_2021} that $\L$ and $f \L f$ are singularly equivalent of Morita type with level.
\end{proof}
%%%%%%%%%%%%%%%%%%%%%%%%%%%%%%%%%%%%%%%%%%%%%%%%%%%%%%

%
%%%%%%%%%%%%%%%      statement     %%%%%%%%%%%%%%%%%%%
% 
\begin{rem}
{\rm
Let $k$ be a field.
The statement of Lemma \ref{claim_35} still holds for finite dimensional connected Nakayama $k$-algebras $\L$ such that the semisimple quotients $\L/J(\L)$ are separable over $k$. 
As such $k$-algebras, there are finite dimensional connected Nakayama $k$-algebras given by quivers with relations.
}\end{rem}
%%%%%%%%%%%%%%%%%%%%%%%%%%%%%%%%%%%%%%%%%%%%%%%

Using a result of  Asashiba \cite{Asa99}, we can classify $\CN$ up to singular equivalence of Morita type with level.
We notice that $\CN$ is not closed under singular equivalence of Morita type with level (see Example \ref{example_3}).

%
%%%%%%%%%%%%%%%      statement     %%%%%%%%%%%%%%%%%%%
%
\begin{theo}\label{claim_33}  
The algebras $k$ and $k Z_{e} / R^{N}$ with $e\geq 1$ and $N\geq 2$ form  a complete set of representatives of pairwise different equivalence classes  of finite dimensional connected Nakayama $k$-algebras under singular equivalence of Morita type with level. 
\end{theo}

\begin{proof}
Let $\L$ be in $\CN$.
Assume that $\gldim\L < \infty$.
Since $\gldim\L =\projdim_{\Le}\L$ by \cite[Section 1.5]{Happel89}, Lemma \ref{claim_31} implies that $\L$ and $k$ are singularly equivalent of Morita type with level.
On the other hand, if $\gldim\L = \infty$, then Lemma \ref{claim_35} implies that there exist integers $e \geq 1$ and $N \geq 2$ such that  $\L$ and $k Z_{e} / R^{N}$ are singularly equivalent of Morita type with level.

We claim that $k Z_{e} / R^{N}$ and $k Z_{e^\prime} / R^{N^\prime}$ are singularly equivalent of Morita type with level if and only if $e = e^\prime$ and $N= N^\prime$.
Assume that $k Z_{e} / R^{N}$ and $k Z_{e^\prime} / R^{N^\prime}$ are singularly equivalent of Morita type with level.
Then these are stably equivalent of Morita type by Lemma \ref{claim_32}.
Since $k Z_{e} / R^{N}$ and $k Z_{e^\prime} / R^{N^\prime}$ are representation-finite self-injective algebras, \cite[Theorem in Section 2.2]{Asa99} implies that $e = e^\prime$ and $N= N^\prime$. 
The converse is clear.

Since $\gldim\L<\infty$ if and only if $\D_{\rm sg}(\L)=0$, there exists no singular equivalence of Morita type with level between $k$ and  $k Z_{e} / R^{N}$ for all $e \geq 1$ and  $N \geq 2$.
This completes the proof.
\end{proof}
%%%%%%%%%%%%%%%%%%%%%%%%%%%%%%%%%%%%%%%%%%%%%%%%%%%%%%

It was proved by Erdmann and Holm \cite[Lemma in Section 4.2]{ErdHolm99} that $k Z_{e} / R^{N}$ is periodic for all $e \geq 1$ and $N \geq 2$. 
Hence Theorems \ref{claim_15} and  \ref{claim_33} imply that an algebra $\L$ in $\CN$ is eventually periodic if and only if $\gldim\L= \infty$.

For each $\L \in \CN$, its {\it resolution quiver} $R(\L)$ can be used to check whether $\gldim\L$ is finite or not.  
Let us recall some related notions and facts from \cite{ARS95,Ringel_2013,Shen_2017}. 
Let $\L$ be in $\CN$, and let $\{T_1, T_2, \ldots, T_e\}$ be a complete set of pairwise non-isomorphic simple $\L$-modules such that $J(\L) P_i$ is a quotient of $P_{i+1}$ for $1 \leq i < e$ and $J(\L) P_e$ is a quotient of $P_{1}$, where $P_i$ is a projective cover of $T_i$. 
We denote by $c_i$ the composition length of  $P_i$.
The sequence $(c_1, c_2, \ldots, c_e)$ is called the {\it admissible sequence} for $\L$ (see \cite[Section IV.2]{ARS95} for more details).
Observe that $\L$ has a simple projective module if and only if $c_e = 1$. 
In this case, the global dimension of $\L$ is always finite. 
Following \cite{Gustafson_1985,Shen_2017}, we define a map $f_{\L}: \{1, 2, \ldots, e\} \rightarrow \{1, 2, \ldots, e\}$ such that  $f_{\L}(i) -(c_i+i)$ is divided by $e$.
The {\it resolution quiver} $R(\L)$ of $\L$ is defined in the following way: the vertices are $1, 2, \ldots, e$, and there exists an arrow from $i$ to $f_{\L}(i)$. 
We remark that $R(\L)$ is the same as that in \cite{Shen_2017}.
By definition, each connected component of $R(\L)$ contains a unique cycle (see \cite[page 243]{Ringel_2013}).
For a cycle $C$ in $R(\L)$ with vertices $T_1, T_2, \ldots, T_r$, we define the {\it weight} of $C$ by $w(C) := \sum_{i=1}^{r} c_i / e$.
Then \cite[the proof of Lemma 2.2]{Shen_2014} and \cite[Proposition 1.1]{Shen_2014} show that $w(C)$ is an integer and that all cycles in $R(\L)$ have the same weight, respectively.
For this reason,  $w(\L):= w(C)$ is called the {\it weight} of $\L$.

It is known that $R(\L)$ characterizes some homological properties on the algebra $\L$. 
For instance, Ringel \cite{Ringel_2013} used  $R(\L)$ to determine whether $\L$ is Gorenstein or not and  whether it is CM-free or not; and Shen \cite{Shen_2017} used $R(\L)$ to determine  whether the global dimension of $\L$ is finite or not.
Recall from \cite[Proposition 1.1]{Shen_2017} that $\gldim\L<\infty$ if and only if $R(\L)$ is connected and $w(\L)=1$. 
Therefore, we obtain the following corollary, which is the main result of this subsection.

%
%%%%%%%%%%%%%%%      statement     %%%%%%%%%%%%%%%%%%%
%
\begin{cor} \label{claim_30}   
Let $\L$ be a finite dimensional connected Nakayama $k$-algebra, where $k$ is an algebraically closed field. 
Then the following conditions are equivalent.
\begin{enumerate}
    \item $\L$ is eventually periodic.
    \item The global dimension of $\L$ is infinite.
    \item The resolution quiver of $\L$ is not connected, or the weight of $\L$ is not equal to $1$.
\end{enumerate}
\end{cor}  
%%%%%%%%%%%%%%%%%%%%%%%%%%%%%%%%%%%%%%%%%%%%%%%%%%%%%%

We end this subsection by providing examples of eventually periodic Nakayama algebras.

%
%%%%%%%%%%%%%%%      statement     %%%%%%%%%%%%%%%%%%%
%
\begin{example} 
{\rm
\begin{enumerate}
\item \label{example_1}
Let $\L$ be a connected Nakayama algebra whose admissible sequence is given by $(13,13,12,12,12)$. 
Note that it can be found in \cite[Example in Introduction]{Ringel_2013}.
The resolution quiver  of $\L$ is as follows.
\begin{align*}
R(\L):\quad\xymatrix@C=10mm{
1 \ar@<0.6ex>[r] & 4 \ar@<0.6ex>[l] & 3 \ar[r] &5 \ar@<0.6ex>[r] & 2 \ar@<0.6ex>[l] 
}\end{align*}
Since $R(\L)$ is disconnected, the algebra $\L$ is eventually periodic by Corollary \ref{claim_30}.

\item  \label{example_2}
Let $\L^\prime$ be a connected Nakayama algebra with admissible sequence $(7,6,6,5)$, which appears in \cite[Example 5.4]{Shen_2015}.
The resolution quiver of $\L^\prime$ is given by
\begin{align*}
R(\L^\prime):\quad\xymatrix@C=10mm{
3 \ar[r] & 1 \ar@<0.6ex>[r] & 4 \ar@<0.6ex>[l] & 2 \ar[l]
}\end{align*}
Since $w(\L^\prime) = 3 \not= 1$, Corollary \ref{claim_30} shows that $\L^\prime$ is eventually periodic.
\end{enumerate} 
}\end{example} 
%%%%%%%%%%%%%%%%%%%%%%%%%%%%%%%%%%%%%%%%%%%%%%%%%%%%%%

%%%%%%%%%%%%%%%%%%%%%%%%%%%%%%%%%%%%
%           ↑ Section 
%%%%%%%%%%%%%%%%%%%%%%%%%%%%%%%%%%%%

%%%%%%%%%%%%%%%%%%%%%%%%%%%%%%%%%%%%
%           Acknowledgments
%%%%%%%%%%%%%%%%%%%%%%%%%%%%%%%%%%%%
\section*{Acknowledgments}  
The author would like to express his appreciation to the referee for reading the manuscript of the paper carefully and for valuable comments.
The author would like to thank Professor Katsunori Sanada, Professor Ayako Itaba and Professor Tomohiro Itagaki for their helpful discussions and for their valuable suggestions and comments for improving the manuscript of the paper.
%%%%%%%%%%%%%%%%%%%%%%%%%%%%%%%%%%%%

\bibliographystyle{plain}

\bibliography{ref}

\end{document}